\documentclass[11pt]{amsart}

\usepackage{amsfonts, amsmath, amscd}
\usepackage[psamsfonts]{amssymb}

\usepackage{amssymb}

\usepackage{pb-diagram}

\usepackage[all,cmtip]{xy}

\usepackage[usenames]{color}

\newtheorem{theorem}{Theorem}[section]
\newtheorem{lemma}[theorem]{Lemma}
\newtheorem{corollary}[theorem]{Corollary}

\newtheorem{proposition}[theorem]{Proposition}

\newtheorem{example}[theorem]{Example}

\numberwithin{equation}{section}

\newcommand{\CC}{C_k}
\newcommand{\NN}{\mathbb{N}}

\newcommand{\GG}{\mathfrak{G}}

\newcommand{\UU}{\mathcal{U}}

\newcommand{\w}{\omega}

\newcommand{\TTT}{\mathcal{T}}

\newcommand{\KK}{\mathcal{K}}

\newcommand{\VV}{\mathbb{V}}

\newcommand{\AAA}{\mathcal A}
\newcommand{\IR}{\mathbb{R}}

\newcommand{\e}{\varepsilon}
\newcommand{\cl}{\mathrm{cl}}

\renewcommand{\phi}{\varphi}

\newcommand{\U}{\mathcal U}

\newcommand{\supp}{\mathrm{supp}}
\newcommand{\conv}{\mathrm{conv}}

\input xy
\xyoption{all}

\title[Locally convex properties of  free locally convex spaces]{Locally convex properties of \\ free locally convex spaces}

\author{S. Gabriyelyan}
\address{Department of Mathematics, Ben-Gurion University of the Negev, Beer-Sheva, P.O. 653, Israel}
\email{saak@math.bgu.ac.il}



\begin{document}

\begin{abstract}
Let $L(X)$ be the free locally convex space  over a Tychonoff space $X$. We show that the following assertions are equivalent: (i) $L(X)$ is $\ell_\infty$-barrelled, (ii) $L(X)$ is $\ell_\infty$-quasibarrelled, (iii) $L(X)$ is $c_0$-barrelled, (iv) $L(X)$ is $\aleph_0$-quasibarrelled, and (v) $X$ is a $P$-space. If $X$ is a non-discrete metrizable space, then $L(X)$ is  $c_0$-quasibarrelled but it is neither  $c_0$-barrelled nor $\ell_\infty$-quasibarrelled. We prove that $L(X)$ is a $(DF)$-space iff $X$ is a countable discrete space.
We show that there is a countable Tychonoff space $X$ such that  $L(X)$ is a  quasi-$(DF)$-space but is  not a $c_0$-quasibarrelled space. For each non-metrizable compact space  $K$, the space $L(K)$ is a $(df)$-space but is not a quasi-$(DF)$-space. If $X$ is a $\mu$-space, then $L(X)$ has the Grothendieck property iff every compact subset of $X$ is finite. We show that $L(X)$ has the Dunford--Pettis property for every Tychonoff space $X$. If $X$ is a sequential $\mu$-space (for example, metrizable), then $L(X)$ has the sequential  Dunford--Pettis property iff $X$ is discrete.
\end{abstract}

\maketitle



\section{Introduction}


We consider three types of locally convex properties (all relevant definitions are given in Section \ref{seq:Main}): weak barrelledness conditions, $(DF)$-like properties and the Dunford--Pettis type properties.

Weak barrelledness concepts are the cornerstone in the study of general locally convex spaces (lcs for short) and have been considered by many authors.
These concepts were examined in particular for different classes of function spaces.  Denote by $\CC(X)$ and $C_p(X)$  the space $C(X)$ of real-valued continuous functions on a  Tychonoff (=completely regular and Hausdorff) space $X$ endowed with the compact-open topology and the pointwise topology, respectively.
Nachbin \cite{Nachbin} and Shirota \cite{Shirota} showed that $\CC(X)$ is barrelled if and only if $X$ is a $\mu$-space. In \cite{Buch-Schmets}, Buchwalter and Schmets proved that $C_p(X)$ is barrelled if and only if every functionally bounded subset of $X$ is finite. Warner \cite{Warner} characterized quasibarrelled spaces $\CC(X)$. It is well known that $C_p(X)$ is quasibarrelled for every Tychonoff space $X$.

The concept of $\aleph_0$-(quasi)barrelledness appears for the first time in Husain \cite{Husain}, and actually it has been already considered by Grothendieck in \cite{Grot-54}. De Wilde and Houet \cite{Wilde-Houet} and Levin and Saxon \cite{Lev-Sax} introduced and studied $\ell_\infty$-(quasi)barrelled spaces, and $c_0$-(quasi)barrelled spaces were treated by Webb \cite{Webb-c0}. Buchwalter and Schmets \cite{Buch-Schmets} showed that $\CC(X)$ is $c_0$-barrelled if and only if it is $\ell_\infty$-barrelled. K\c{a}kol, Saxon and Todd \cite{KST-1} constructed an $\ell_\infty$-barrelled space $\CC(X)$ which is not $\aleph_0$-barrelled. For further results and historical remarks we refer the reader to the classical books \cite{Jar} and \cite{PB} and the articles \cite{KST-1,KST}.

An important subclass of $\aleph_0$-quasibarrelled spaces is the class of $(DF)$-spaces introduced by Grothendieck \cite{Grot-54}. A wider class of $(df)$-spaces was defined by Jarchow \cite{Jar}. An lcs $(E,\tau)$ is called a {\em $(DF)$-space} (a {\em almost $(DF)$-space} or a {\em $(df)$-space}) if it has a fundamental bounded sequence and is $\aleph_0$-quasibarrelled ($\ell_\infty$-quasibarrelled or $c_0$-quasibarrelled, respectively).
The strong dual of any Fr\'{e}chet space is a $(DF)$-space. Every separable almost $(DF)$-space is a duasibarrelled $(DF)$-space by Proposition 12.5.4 of \cite{Jar}. In \cite{Warner}, Warner proved that $\CC(X)$ is a $(DF)$-space if and only if each countable union of compact sets in $X$ is relatively compact. Morris and Wulbert \cite{Mor-Wul} showed that there are $(DF)$-spaces $\CC(X)$ which are not $\aleph_0$-barrelled.  An example of a $(df)$-space $\CC(X)$ which is not a $(DF)$-space is given in \cite[Example~3]{KST}. For  very nice expositions of $(DF)$-like locally convex spaces we refer the reader to Chapter 12 of \cite{Jar} and Chapter 8 of \cite{PB}. In \cite{FGK-quasi-DF}, we introduced and studied the class of quasi-$(DF)$-spaces and constructed an example of a quasi-$(DF)$-space which is not a $(DF)$-space. The diagram below shows  relationships between the discussed weak barrelledness conditions
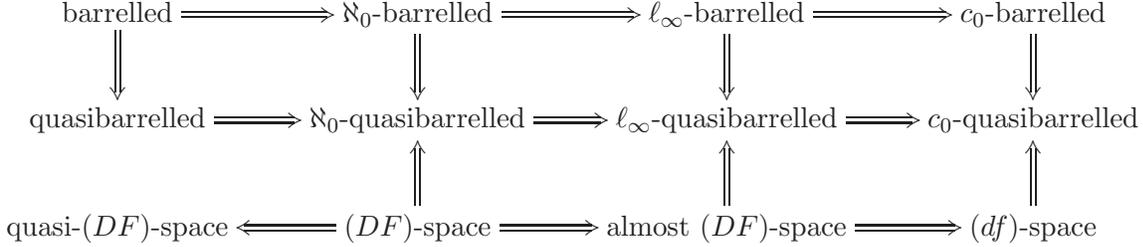
\[
\xymatrix{
\mbox{barrelled} \ar@{=>}[r]\ar@{=>}[d]   &   \mbox{$\aleph_0$-barrelled} \ar@{=>}[r] \ar@{=>}[d] &  \mbox{$\ell_\infty$-barrelled}  \ar@{=>}[r]\ar@{=>}[d] &  \mbox{$c_0$-barrelled} \ar@{=>}[d]\\
\mbox{quasibarrelled} \ar@{=>}[r]   & \mbox{$\aleph_0$-quasibarrelled} \ar@{=>}[r] & \mbox{$\ell_\infty$-quasibarrelled} \ar@{=>}[r] & {\mbox{$c_0$-quasibarrelled}} \\
\mbox{quasi-$(DF)$-space} & \mbox{$(DF)$-space} \ar@{=>}[r] \ar@{=>}[l] \ar@{=>}[u] &  \mbox{almost $(DF)$-space} \ar@{=>}[r] \ar@{=>}[u] & \mbox{$(df)$-space}  \ar@{=>}[u]
}
\]

Recall that an lcs $E$ is said to have  the {\em Grothendieck property}  if every weak-$\ast$ convergent sequence in the strong dual $E'_\beta$ is weakly convergent. We proved in \cite{GK-DP} that $C_p(X)$ has the Grothendieck property if and only if every functionally bounded subset of $X$ is finite, and if $X$ is a sequential space, then $\CC(X)$ has the Grothendieck property if and only if $X$ is discrete.

Following Grothendieck \cite{Grothen}, an lcs $E$ is said to have  the {\em Dunford--Pettis property} ($(DP)$ {\em property} for short) if every continuous linear operator $T$ from $E$ into a quasi-complete locally convex space $F$, which transforms bounded sets of $E$ into relatively weakly compact subsets of $F$, also transforms absolutely convex weakly compact subsets of $E$ into relatively compact subsets of $F$. In \cite{GK-DP}, we proved that $C_p(X)$ has the $(DP)$ property for every Tychonoff space $X$ and showed that $\CC(X)$ has the $(DP)$ property if $X$ is  hemicompact or a cosmic space.

Grothendieck proved in \cite[Proposition~2]{Grothen} that a Banach space $E$ has the $(DP)$ property if and only if  given weakly null sequences $\{ x_n\}_{n\in\NN}$ and $\{ \chi_n\}_{n\in\NN}$ in $E$ and the Banach dual $E'$ of $E$, respectively, then $\lim_n \chi_n(x_n)=0$. He used this result to show that every Banach space $C(K)$ has the $(DP)$ property, see \cite[Th\'{e}or\`{e}me~1]{Grothen}. Extending this result to locally convex spaces and following \cite{Gabr-free-resp}, we  consider the following ``sequential'' version of the $(DP)$ property: an lcs $E$ is said to have  the {\em sequential Dunford--Pettis property} ($(sDP)$ {\em property})  if  given weakly null sequences $\{ x_n\}_{n\in\NN}$ and $\{ \chi_n\}_{n\in\NN}$ in $E$ and the strong dual $E'_\beta$ of $E$, respectively, then $\lim_n \chi_n(x_n)=0$. In \cite{GK-DP}, we showed that $C_p(X)$ has the $(sDP)$ property for every Tychonoff space $X$. If $X$ is an ordinal space or a locally compact paracompect space, then $\CC(X)$ has the $(sDP)$ property (see \cite{GK-DP}).

The aforementioned results motivate to consider weak barrelledness conditions, the Grothendieck property and $(DP)$-type properties in other important classes of locally convex spaces. One of such classes is the class of free locally convex spaces. Following \cite{Mar}, the {\em  free locally convex space}  $L(X)$ on a Tychonoff space $X$ is a pair consisting of a locally convex space $L(X)$ and  a continuous map $i: X\to L(X)$  such that every  continuous map $f$ from $X$ to a locally convex space  $E$ gives rise to a unique continuous linear operator ${\bar f}: L(X) \to E$  with $f={\bar f} \circ i$. The free locally convex space $L(X)$ always exists and is essentially unique.

We consider the following problem: {\em Characterize Tychonoff space $X$ for which the free lcs $L(X)$ satisfies some of weak barrelledness conditions}. 
Let us recall the results which are known up to now. It is proved in Theorem 5 of \cite{FKS-1} that $L(X)$  is quasibarrelled if and only if $L(X)$ is barrelled if and only if $X$ is discrete. This result was essentially strengthen  in \cite{Gabr-L(X)-Mackey}, where we proved that $L(X)$ is a Mackey space if and only if $X$ is discrete. This somewhat surprising result shows that, for a non-discrete $X$, the Mackey topology of  $L(X)$ induced on $X$ is strictly finer than the original topology of $X$. In \cite[Theorem~2.3]{Saxon-14}, Saxon proved that $L(X)$ is $\aleph_0$-barrelled if and only if $X$ is discrete. Ferrando, K\c{a}kol and Saxon proved in  \cite[Theorem~6]{FKS-1} that $L(X)$ is $\ell_\infty$-barrelled if and only if $X$ is a $P$-space. Recall that a topological space $X$ is called a {\em $P$-space} if every countable intersection of open sets is open. In  Theorem \ref{t:L(X)-c0-barrelled} below, we characterize  other weak barrelledness conditions from two upper rows of the diagram (we include (i) to Theorem \ref{t:L(X)-c0-barrelled} because our proof is direct and short).

\begin{theorem} \label{t:L(X)-c0-barrelled}
For a Tychonoff space $X$, the following assertions are equivalent:
\begin{enumerate}
\item[{\rm (i)}] {\rm (\cite{FKS-1})} $L(X)$ is $\ell_\infty$-barrelled;
\item[{\rm (ii)}] $L(X)$ is $\ell_\infty$-quasibarrelled;
\item[{\rm (iii)}] $L(X)$ is $c_0$-barrelled;
\item[{\rm (iv)}] $L(X)$ is $\aleph_0$-quasibarrelled;
\item[{\rm (v)}] $X$ is a $P$-space.
\end{enumerate}
\end{theorem}

We shall say that  $X$  is  a  {\em sequentially Ascoli space} if every convergent sequence in $\CC(X)$ is equicontinuous. Every metrizable space is a $\mu$-space and sequentially Ascoli (by the Ascoli theorem \cite{Eng}). In the next theorem we give  a  sufficient condition on $X$ such that $L(X)$ is a $c_0$-quasibarrelled space.  

\begin{theorem} \label{t:L(X)-c0-quasibarrelled}
Let $X$ be a $\mu$-space and a sequentially Ascoli space. Then $L(X)$ is a $c_0$-quasibarrelled space.
\end{theorem}

Since every metrizable $P$-space is discrete, Theorems  \ref{t:L(X)-c0-barrelled} and   \ref{t:L(X)-c0-quasibarrelled} immediatelly imply
\begin{corollary} \label{c:L(X)-c0-barrelled-c0-quasibarrelled}
If $X$ is a non-discrete metrizable space, then $L(X)$ is  $c_0$-quasibarrelled but it is neither  $c_0$-barrelled nor $\ell_\infty$-quasibarrelled.
\end{corollary}

We provide also Tychonoff spaces $X$ for which the space $L(X)$ is not $c_0$-quasibarrelled, see Proposition \ref{p:L(X)-not-c0-quasibarrelled} and Example \ref{exa:non-c0-quasibarrelled-1}.

Now we consider $(DF)$-like properties from the lower row of the diagram.
The following result shows that the condition  for $L(X)$ of being a (almost) $(DF)$-space is too strong.
\begin{theorem} \label{t:L(X)-DF-space}
For a Tychonoff space $X$ the following assertions are equivalent:
\begin{enumerate}
\item[{\rm (i)}] $L(X)$ is a $(DF)$-space;
\item[{\rm (ii)}] $L(X)$ is an almost $(DF)$-space;
\item[{\rm (iii)}]  $X$ is a countable discrete space.
\end{enumerate}
\end{theorem}

In \cite[Example 4.10]{FGK-quasi-DF} we constructed  a countable Tychonoff space $X$ such that $L(X)$ is a quasi-$(DF)$-space but not a $(DF)$-space.
In Example \ref{exa:non-c0-quasibarrelled} below we considerably strengthen the conclusion of this example by showing that $L(X)$ is not even a $c_0$-quasibarrelled space. On the other hand, in Proposition \ref{p:df-non-quasi-DF}, we prove that if $K$ is a non-metrizable compact space, then $L(K)$ is a $(df)$-space but not a quasi-$(DF)$-space. Therefore the notions of quasi-$(DF)$-spaces and $(df)$-spaces are different in the class of free locally convex spaces.

Below we characterize the Grothendieck property for free lcs over $\mu$-spaces.

\begin{theorem} \label{t:L(X)-Grothendieck}
Let $X$ be a $\mu$-space. Then $L(X)$ has the Grothendieck property if and only if every compact subset of $X$ is finite.
\end{theorem}

It turns out that free locally convex spaces always have the $(DP)$-property.

\begin{theorem} \label{t:L(X)-DP-property}
For every Tychonoff space $X$, the space $L(X)$ has the $(DP)$-property.
\end{theorem}

Albanese, Bonet and Ricker (\cite[Corollary~3.4]{ABR}) generalized the above mentioned Grothendieck's result by  proving  that the $(DP)$ property and the $(sDP)$ property coincide also for the class of Fr\'{e}chet spaces (or, even more generally, for  strict $(LF)$-spaces). In \cite[Proposition~3.3]{ABR} they showed  that every barrelled quasi-complete space with the  $(DP)$ property has also the  $(sDP)$ property. Since $L(X)$ is a Mackey space only for the trivial case of discrete $X$ (see \cite{Gabr-L(X)-Mackey}), one can expect that the $(DP)$ property differs from the $(sDP)$ property in the class of free lcs. This is indeed so as the following theorem shows.

\begin{theorem} \label{t:L(X)-sDP-property}
Let $X$ be a sequential $\mu$-space. Then $L(X)$ has the $(sDP)$-property if and only if $X$ is discrete.
\end{theorem}

In particular, if $X$ is a non-discrete metrizable space, then $L(X)$ has the $(DP)$-property but it does not have the $(sDP)$-property.


%
The paper is organized as follows. In Section \ref{seq:Free}  we give, among others, a new description of the topology of free locally convex spaces (Theorem \ref{t:topology-V(X)}) and characterize bounded subsets of free lcs (Proposition \ref{p:bounded-in-L(X)}). Theorems \ref{t:L(X)-c0-barrelled}-\ref{t:L(X)-sDP-property} are proved in Section \ref{seq:Main}.


\section{Description of the topology of free topological vector spaces and free locally convex spaces} \label{seq:Free}


We start from some necessary definitions and notations.  Set $\NN:=\{ 1,2,\dots\}$, $\w:=\{ 0,1,\dots\}$ and  $\IR_{>0}=(0,\infty)$.
The closure of a subset $A$ of a Tychonoff space $X$ is denoted by $\overline{A}$ or $\cl(A)$. 
The support of a function $f\in C(X)$ is denoted by $\supp(f)$.  The characteristic function of a subset $E$ of $X$ is denoted by  $\mathbf{1}_E$. Recall that the sets
\[
[K;\e] :=\{ f\in C(X): |f(x)|<\e \; \mbox{ for all } x\in K\},
\]
where $K$ is a compact subset of $X$ and $\e>0$, form a base at zero of the compact-open topology $\tau_k$ of $\CC(X)$.
A subset $A$ of $X$ is called {\em functionally bounded in $X$} if every $f\in C(X)$ is bounded on $A$; $X$ is said to be a {\em $\mu$-space} if every functionally bounded set in $X$ has compact closure. It is well known that every Dieudonn\'{e} complete space is a $\mu$-space. Recall that a Tychonoff space $X$ is {\em Dieudonn\'{e} complete} if the universal uniformity $\U_X$ on $X$ is complete.  For numerous characterizations of Dieudonn\'{e} complete spaces see Section 8.5.13 of \cite{Eng}. The {\em Dieudonn\'{e} completion} $\mu X$ of $X$ is the completion of the uniform space $(X,\U_X)$.

Let $E$ be a locally convex space. The dual space $E'$  of $E$ endowed with the  weak-$\ast$ topology  $\sigma(E',E)$ and the strong topology $\beta(E',E)$ on $E'$ is denoted by $E'_w$ and $E'_\beta$, respectively. The polar of a subset $A$ of $E$ is denoted by
\[
A^\circ :=\{ \chi\in E': |\chi(x)|\leq 1 \, \mbox{ for all } x\in A\}.
\]

Following \cite{GM}, the {\em free topological vector space} $\VV(X)$ over a Tychonoff space $X$ is a pair consisting of a topological vector space $\VV(X)$ and a continuous map $i=i_X: X\to \VV(X)$ such that every continuous map $f$ from $X$ to a topological vector space $E$ gives rise to a unique continuous linear operator ${\bar f}: \VV(X) \to E$ with $f={\bar f} \circ i$. Theorem 2.3 of \cite{GM} shows that for all Tychonoff spaces $X$, $\VV(X)$ exists, is unique up to isomorphism of topological vector spaces and is Hausdorff. For every Tychonoff space $X$, the set $X$ forms a Hamel basis  and  the map $i$ is a topological embedding for both spaces $\VV(X)$ and $L(X)$ (see \cite{GM,Rai}), so we shall identify $x\in X$ with its image $i(x)$.
If $D$ is a countably infinite discrete space, then  $L(D)$ is topologically isomorphic to the direct sum of a countably infinite family of the real line $\IR$ and denoted by $\phi$.

Denote by $\pmb{\mu}_X$ and $\pmb{\nu}_X$ the topology of $\VV(X)$ and $L(X)$, respectively. So $\VV(X) =(\VV_X, \pmb{\mu}_X)$ and $L(X)=(\VV_X, \pmb{\nu}_X)$, where $\VV_X$ is a vector space with a basis $X$.
In Theorem 1' of \cite{Rai}, Ra\u{\i}kov obtained the following description of the topology $\pmb{\nu}_X$ of $L(X)$.
\begin{theorem}[\cite{Rai}] \label{t:topology-L(X)-Raikov}
Let $X$ be a Tychonoff space. Then the sets of the form
\[
S^\circ :=\{ \chi\in L(X): |\chi(f)|\leq 1 \mbox{ for every } f\in S\},
\]
where $S$ is a pointwise bounded and equicontinuous subset of $C(X)$, form a basis of neighborhoods of $\pmb{\nu}_X$.
\end{theorem}
A description of the topology $\pmb{\mu}_X$ of $\VV(X)$ for a uniform space $X$ is given in Section 5 of \cite{BL}. In the next theorem we describe the topologies $\pmb{\mu}_X$ and $\pmb{\nu}_X$ for any Tychonoff space $X$. First we explain our notations and construction.

Let $X$ be a Tychonoff space. Take a balanced and absorbent neighborhood $W$ of zero in $\VV(X)$ and choose  a sequence $\{ W_n\}_{n\in\NN}$ of balanced and absorbent neighborhoods of zero in $\VV(X)$ such that $W_1 +W_1 \subseteq W$ and $W_{n+1}+W_{n+1} \subseteq W_n$ for every $n\in \NN$.   For every $n\in\NN$ and each $x\in X$, choose a function $\phi_n \in \IR^X_{>0}$ such that $W_n$ contains a subset of the form
\[
S_n :=\bigg\{ t x: x \in X \mbox{ and } |t| \leq \frac{1}{\phi_n(x)} \bigg\}.
\]
Then $W$ contains a subset of the form
\begin{equation} \label{equ:topology-V-L-1}
\begin{aligned}
\sum_{n\in\NN} \frac{1}{\phi_n} X & =\sum_{n\in\NN} S_n := \bigcup_{m\in\NN} \big( S_1 +\cdots +S_m\big)\\
& =\bigcup_{m\in\NN} \left\{ \sum_{n=1}^m t_n x_n: x_n \in X \mbox{ and } |t_n| \leq \frac{1}{\phi_n(x_n)}  \mbox{ for all } n\leq m \right\}.
\end{aligned}
\end{equation}
If the space $X$ is discrete,  Protasov showed in \cite{Prot} that the family $\mathcal{N}_X$ of all subsets of $\VV_X$ of the form $\sum_{n\in\NN} \frac{1}{\phi_n} X$ is a base at zero $\mathbf{0}$ for $\pmb{\mu}_X$, and the family $\mathcal{\widehat{N}}_X :=\{ \conv(V): V\in \mathcal{N}_X\}$ is a base  at $\mathbf{0}$ for $\pmb{\nu}_X$ (where $\conv(V)$ is the convex hull of $V$). If $X$ is arbitrary, observe that  every $W_n$ defines the entourage $V_n :=\{ (x,y): x-y \in W_n\}$ of the universal uniformity $\UU_X$  of the uniform space $X$. Therefore $W$  contains a subset of the form
\begin{equation} \label{equ:topology-V-L-2}
\sum_{n\in\NN} V_{n} := \bigcup_{m\in\NN} \left\{ \sum_{n=1}^m t_n (x_n-y_n): |t_n|\leq 1 \mbox{ and } (x_n,y_n)\in  V_{n} \mbox{ for all } n\leq m\right\}.
\end{equation}
Combining  (\ref{equ:topology-V-L-1}) and (\ref{equ:topology-V-L-2}) we obtain that every balanced and absorbent neighborhood $W$ of zero in $\VV(X)$ contains a subset of the form
$
\sum_{n\in\NN} V_{n} + \sum_{n\in\NN} \frac{1}{\phi_n} X,
$
where $\{V_{n}\}_{n\in\NN}\in \UU_X^\NN$ and $\{\phi_n\}_{n\in\NN}\in \IR^X_{>0}$. It turns out that the converse is also true.

\begin{theorem} \label{t:topology-V(X)}
The family
\[
\mathcal{B}=\left\{ \sum_{n\in\NN} V_{n} + \sum_{n\in\NN} \frac{1}{\phi_n} X : \{V_{n}\}_{n\in\NN}\in \UU_X^\NN ,\; \{\phi_n\}_{n\in\NN}\in \IR^X_{>0}\right\}
\]
forms a neighbourhood base at zero of $\VV(X)$, and the family
\[
\mathcal{B}_L=\{ \conv(W): W\in\mathcal{B}\},
\]
where $\conv(W)$ is the convex hull of $W$, is a base at zero of $L(X)$.
\end{theorem}

\begin{proof}
We prove the theorem in two steps.

\smallskip
{\em Step 1. We claim that the family $\mathcal{B}$ is a base of some vector topology $\TTT$ on $\VV_X$.} Indeed, by construction, each set $W\in\mathcal{B}$ is balanced and absorbent. It is clear that $\mathcal{B}$ is a filterbase at zero. So, by Theorem 4.5.1 of \cite{NaB}, we have to check only that, for every $W=\sum_{n\in\NN} V_{n} + \sum_{n\in\NN} \frac{1}{\phi_n} X\in\mathcal{B}$, there is a $W'=\sum_{n\in\NN} V'_{n} + \sum_{n\in\NN} \frac{1}{\phi'_n} X\in\mathcal{B}$ such that $W' +W' \subseteq W$. For every $n\in\NN$, choose $V'_n\in \UU_X$ such that $V'_{n} \subseteq V_{2n-1} \cap V_{2n}$ and $\phi'_n\in\IR^X_{>0}$ such that $\phi'_n \geq\max\{\phi_{2n-1}, \phi_{2n}\}$. Then, for every $m\in\NN$, we obtain the following: if  $|t_n|,|s_n|\leq 1 $ and $(x_n,y_n), (u_n,v_n)\in  V'_{n}$, then
\[
\begin{split}
\sum_{n=1}^m t_n (x_n-y_n) & + \sum_{n=1}^m s_n (u_n-v_n) \\
& = t_1 (x_1-y_1) + s_1 (u_1-v_1)+\cdots + t_m (x_m-y_m) +s_m (u_m-v_m)\\
& \in \left\{ \sum_{n=1}^{2m} t_n (x_n-y_n): |t_n|\leq 1 \mbox{ and } (x_n,y_n)\in  V_{n} \mbox{ for all } n\leq 2m\right\},
\end{split}
\]
and if $|t_n|, |s_n| \leq \frac{1}{\phi'_n(x)}$ and $x_n,y_n\in X$, then
\[
\begin{split}
\sum_{n=1}^m t_n x_{n}  + \sum_{n=1}^m s_n y_{n}  & = t_1 x_{1} + s_1 y_{1} +\cdots + t_m x_{m} +s_m y_{m}\\
& \in \left\{ \sum_{n=1}^{2m} t_n x_{n}: x_{n} \in X \mbox{ and } |t_n| \leq \frac{1}{\phi_n(x_n)}  \mbox{ for all } n\leq 2m\right\}.
\end{split}
\]
These inclusions easily imply $W' + W' \subseteq W$.

\smallskip
{\em Step 2. We claim that $\TTT=\pmb{\mu}_X$.} Indeed, if $x\in X$ and $W=\sum_{n\in\NN} V_{n} + \sum_{n\in\NN} \frac{1}{\phi_n} X\in \mathcal{B}$, then $x+W$ contains the neighborhood $V_{1}[x]:=\{ y\in X: (x,y)\in V_{1}\}$ of $x$ in $X$. Hence the identity map $\delta: X\to (\VV_X,\TTT), \delta(x)=x,$ is continuous. Therefore $\TTT\leq \pmb{\mu}_X$ by the definition of $\pmb{\mu}_X$. Below we show that $\TTT\geq \pmb{\mu}_X$.

Given any circled and absorbent neighborhood $U$ of zero in $\pmb{\mu}_X$, choose symmetric neighborhoods $U_0,U_1,\dots$ of zero in $\pmb{\mu}_X$ such that $[-1,1]U_0 +[-1,1]U_0\subseteq U$ and
\[
 [-1,1]U_k + [-1,1]U_k + [-1,1]U_k \subseteq U_{k-1}, \quad \mbox{ for every } k\in\NN.
\]
Since $\UU_X$ is the universal uniformity and $X$ is a subspace of $\VV(X)$ by Theorem 2.3 of \cite{GM}, for every $n\in\NN$, we can choose $V_n\in\UU_X$ such that $y-x\in U_n$ for every $(x,y)\in V_{n}$. For every $n\in\NN$ and each $x\in X$, choose $\lambda(n,x)>0$ such that
\[
[-\lambda(n,x),\lambda(n,x)]x \subseteq U_n,
\]
and set $\phi_n (x):= [1/\lambda(n,x)]+1$. Clearly, $\phi_n \in\IR_{>0}^X$ for every $n\in\NN$. Then, for every $m\in \NN$, we obtain the following: if $|t_n|\leq 1 \mbox{ and } (x_n,y_n)\in  V_{n} \mbox{ for all } n\leq m$, then
\[
\sum_{n=1}^m t_n (x_n-y_n) \in  [-1,1]U_1 + \cdots  + [-1,1]U_m \subseteq U_0,
\]
and if $ |t_n| \leq \frac{1}{\phi_n(x_n)}$ for $n=1,\dots,m$, then
\[
\sum_{n=1}^m t_n x_{n} \in [-1,1]U_1 +\cdots + [-1,1]U_m \subseteq U_0.
\]
Therefore $\sum_{n\in\NN} V_{n} + \sum_{n\in\NN} \frac{1}{\phi_n} X\subseteq U$. Thus $\TTT \geq \pmb{\mu}_X$ and hence $\TTT=\pmb{\mu}_X$.

Finally, the definition of the topology $\pmb{\nu}_X$ of $L(X)$ and Proposition 5.1 of \cite{GM} imply that the family $\mathcal{B}_L$ is a base at zero of $\pmb{\nu}_X$.
\end{proof}

From the definition of $L(X)$ it easily follows the well known fact that the dual space $L(X)'$ of $L(X)$ is linearly isomorphic to the space $C(X)$. Therefore there are two notions of equicontinuity of a subset $S$ of $C(X)$: $S$ is equicontinuous as a set of functions on $X$ and $S$ is equicontinuous as a subset of the dual space of $L(X)$. 

\begin{proposition} \label{p:equicontinuity-C(X)}
Let $X$ be a Tychonoff space. If a subset $S$ of $C(X)$ is equicontinuous as a subset of the dual space of $L(X)$, then $S$ is equicontinuous as a subset of $C(X)$.
\end{proposition}

\begin{proof}
Fix $z\in X$ and $\e>0$. Take a neighborhood $U$ of zero in $L(X)$ such that $S\subseteq U^\circ$. Choose an $n\in\NN$ such that $n> 1/\e$ and a neighborhood $W$ of zero in $L(X)$ such that $nW \subseteq U$. Theorem \ref{t:topology-V(X)} implies that there exists an entourage $V\in\UU_X$ such that $x-y\in W$ for every $(x,y)\in V$. Observe that the set $\mathcal{O}:=\{ x\in X: (x,z)\in V\}$ is a neighborhood of $z$. Then, for every $x\in \mathcal{O}$ and each $f\in S$, we obtain
\[
|f(x)-f(z)|=\frac{1}{n} |n(f(x)-f(z))|= \frac{1}{n} \big| f\big( n(x-z)\big)\big|\leq \frac{1}{n} \cdot 1 <\e.
\]
Thus $S$ is equicontinuous as a subset of $C(X)$. \qed
\end{proof}

Denote by $M_c(X)$ the space of all real regular Borel measures on $X$ with compact support. It is well known that the dual space of $\CC(X)$ is $M_c(X)$, see \cite[Proposition~7.6.4]{Jar}. For every $x\in X$, we denote by $\delta_x \in M_c(X)$ the evaluation map (Dirac measure), i.e. $\delta_x(f):=f(x)$ for every $f\in C(X)$. The total variation norm of a measure $\mu \in M_c(X)$ is denoted by $\| \mu\|$. Denote by $\tau_e$ the polar topology on $M_c(X)$ defined by the family of all equicontinuous  pointwise bounded subsets of $C(X)$.
We shall use the following deep result of Uspenski\u{\i} \cite{Usp2}.

\begin{theorem}[\cite{Usp2}] \label{t:Free-complete-L}
Let $X$ be a Tychonoff space and let $\mu X$ be the Dieudonn\'{e} completion of $X$. Then the completion $\overline{L(X)}$ of $L(X)$ is topologically isomorphic to $\big(M_c(\mu X),\tau_e\big)$.
\end{theorem}
In what follows we shall also identify elements $x\in X$ with the corresponding Dirac measure $\delta_x \in M_c(X)$.
We need the following corollary of Theorem \ref{t:Free-complete-L} noticed in \cite{Gabr-L(X)-Mackey}.
\begin{corollary}[\cite{Gabr-L(X)-Mackey}] \label{p:Ck-Mc-compatible}
Let $X$ be a  Dieudonn\'{e} complete space. Then $(M_c(X),\tau_e)' =\CC(X)$.
\end{corollary}

We need  also the following fact, see \S 5.10 in \cite{NaB}. 
\begin{proposition} \label{p:Ascoli-Free-Ck}
Let $X$ be a Tychonoff space and let $\KK$ be an equicontinuous pointwise bounded subset of $C(X)$. Then the pointwise closure ${\bar A}$ of $A$ is $\tau_k$-compact and equicontinuous.  
\end{proposition}

For  $\chi = a_1 x_1+\cdots +a_n x_n\in L(X)$ with distinct $x_1,\dots, x_n\in X$ and  nonzero $a_1,\dots,a_n\in\IR$, we set
\[
\| \chi\|:=|a_1|+\cdots +|a_n|, \; \mbox{ and } \; \mathrm{supp}(\chi):=\{ x_1,\dots, x_n\},
\]
and recall that
\[
f(\chi)=a_1 f(x_1)+\cdots +a_n f(x_n), \; \mbox{ for every } f\in C(X)=L(X)'.
\]
For $\{ 0\}\not= A\subseteq L(X)$, set $\mathrm{supp}(A):=\bigcup_{\chi\in A} \mathrm{supp}(\chi)$ and $C_A:=\sup\{ \| \chi\|: \chi\in A\}$.
Below we describe bounded subsets of $L(X)$. 

\begin{proposition} \label{p:bounded-in-L(X)}
For a nonzero subset $A$ of $L(X)$ the following assertions are equivalent:
\begin{enumerate}
\item[{\rm (i)}] $A$ is bounded;
\item[{\rm (ii)}] $\supp(A)$ has compact closure in the Dieudonn\'{e} completion $\mu X$ of $X$ and $C_A$ is finite;
\item[{\rm (iii)}] $\supp(A)$ is functionally bounded in $X$ and $C_A$ is finite.
\end{enumerate}
\end{proposition}

\begin{proof}
Observe that a subset $B$ of an lcs $E$ is bounded if and only if its closure $\overline{B}$ in the completion $\overline{E}$ of $E$ is bounded.

(i)$\Rightarrow$(ii) Let $A$ be bounded. By Theorem \ref{t:Free-complete-L}, we have $\overline{L(X)}=(M_c(\mu X),\tau_e)$ and, by Corollary \ref{p:Ck-Mc-compatible}, the topology $\tau_e$ is compatible with the duality $(\CC(\mu X),M_c(\mu X))$. As $\mu X$ is a $\mu$-space, the Nachbin--Shirota theorem implies that $\CC(\mu X)$ is barrelled. Therefore $A$ is a bounded subset of $L(X)$ if and only if its closure $\overline{A}$ in $(M_c(\mu X),\tau_e)$ is equicontinuous and hence if and only if there is a compact subset $K$ of $\mu X$ and $\e>0$ such that $A\subseteq [K;\e]^\circ \cap L(X)$. By the regularity of $\mu X$ and the density of $X$ in $\mu X$, it is easy to see that
\[
\chi = a_1 x_1+\cdots +a_n x_n \in [K;\e]^\circ \cap L(X),
\]
where $x_1,\dots, x_n\in X$ are distinct and $a_1,\dots,a_n$ are nonzero, if and only if $x_1,\dots,x_n\in K$ and $\| \chi\|=|a_1|+\cdots +|a_n|\leq 1/\e$. Therefore, if $A$ is bounded, then $\supp(A)\subseteq K$ and $C_A \leq 1/\e$.

(ii)$\Rightarrow$(i) Let $\overline{\supp(A)}$ be compact in $\mu X$ and $C_A<\infty$. Set $B:= \big[\overline{\supp(A)}; 1/C_A \big]^\circ$. Then $B$ is absolutely convex, equicontinuous and  $\sigma\big(M_c(\mu X),C(\mu X)\big)$-compact by the Alaoglu theorem. Therefore $B$ is compact in the precompact-open topology $\tau_{pc}$ on $M_c(\mu X)$ by Proposition 3.9.8 of \cite{horvath}. Since the compact-open topology on $M_c(\mu X)$ is clearly weaker than $\tau_{pc}$,  Proposition \ref{p:Ascoli-Free-Ck} implies that  $\tau_e\leq \tau_{pc}$, and hence $B$ is a $\tau_e$-compact subset of $M_c(\mu X)$. As $A\subseteq B\cap L(X)$, the above observation implies that $A$ is a bounded subset of $L(X)$.

The equivalence of (ii) and (iii) follows from the well known fact that a subset of $X$ is functionally bounded if and only its closure in $\mu X$ is compact. \qed
\end{proof}

Below we use the following simple lemma.
\begin{lemma} \label{l:bounded-in-P-space}
Every functionally bounded subset of a Tychonoff $P$-space is finite.
\end{lemma}

\begin{proof}
Suppose for a contradiction that there is an infinite functionally bounded subset $A$ of $X$. Then, by Lemma 11.7.1 of \cite{Jar},  there is a one-to-one sequence $\{ a_n\}_{n\in\NN}$ in $A$ and a sequence $\UU=\{ U_n\}_{n\in\NN}$ of open pairwise disjoint subsets of $X$ such that $a_n\in U_n$ for every $n\in\NN$. Since $X$ being a $P$-space is zero-dimensional, we can assume that all $U_n$ are clopen. Let us show that $\UU$ is a discrete family. Indeed, if $x\in X\setminus \bigcup_{n\in\NN} U_n$, then, for every $n\in \NN$, there is a  clopen neighborhood $V_n$ of $x$ such that $V_n\cap U_n=\emptyset$. Set $V=\bigcap_{n\in\NN} V_n$. Then $V$ is a clopen neighborhood of $x$ which does not intersect $\bigcup_{n\in\NN} U_n$. Therefore the function
$
f(x)= \sum_{n\in\NN} n\cdot \mathbf{1}_{U_n}
$
is well-defined and continuous on $X$. Since, by construction, $f$ is unbounded on the functionally bounded set $A$ we obtain a contradiction. \qed
\end{proof}

Following \cite{BG}, a Tychonoff space $X$  is called an {\em  Ascoli space}  if every  compact subset $\KK$ of $\CC(X)$ is evenly continuous. It is noticed in \cite{Gabr-LCS-Ascoli} that a Tychonoff space $X$ is Ascoli if and only if every compact subset of $\CC(X)$ is equicontinuous. By the classical Ascoli theorem \cite[Theorem~3.4.20]{Eng}, every $k$-space is an Ascoli space.  It is clear that every Ascoli space is sequentially Ascoli, but the converse is not true in general as the following proposition shows.

\begin{proposition} \label{p:P-space-Ascoli}
Let $X$ be a non-discrete $P$-space. Then:
\begin{enumerate}
\item[{\rm (i)}] every sequence in $\CC(X)$ is equicontinuous;
\item[{\rm (ii)}] $X$ is  sequentially Ascoli but not Ascoli.
\end{enumerate}
\end{proposition}

\begin{proof}
(i) Let $S=\{ f_n\}_{n\in\w}$ be a sequence in $C(X)$. Fix a point $z\in X$ and $\e>0$. For every $n\in\w$, choose an open neighborhood $U_n$ of $z$ such that $|f_n(x)-f_n(z)|<\e$ for all $x\in U_n$. Set $U:=\bigcap_{n\in\w} U_n$. Then $U$ is a neighborhood of $x$ because $X$ is a $P$-space, and
\[
|f_n(x)-f_n(z)|<\e \; \mbox{ for all } x\in U \mbox{ and } n\in\w.
\]
Thus $S$ is equicontinuous at $z$.

(ii) It follows from (i) that $X$ is a sequentially Ascoli space. To show that $X$ is not Ascoli, we have to find a compact subset of $\CC(X)$ which is not equicontinuous. Fix a non-isolated point $z\in X$. Using the Zorn lemma and zero-dimensionality of $X$, choose an arbitrary maximal (under inclusion) family $\UU=\{U_i: i\in I\}$ of clopen subsets of $X\setminus\{z\}$ such that $U_i\cap U_j=\emptyset$ for all distinct $i,j\in I$. The maximality of $\UU$ implies that $z\in \overline{\bigcup\UU}$. Set $\KK:= \{ 0\} \cup \{ \mathbf{1}_{U_i}: i\in I\}$.

We claim that the family $\KK$ is a compact subset of $\CC(X)$. Indeed, take an arbitrary standard neighborhood $[K;\e]$ of the zero function $0\in \KK$, where $K$ is a compact subset of $X$ and $\e>0$. By Lemma \ref{l:bounded-in-P-space}, $K$ is finite. Since the sets $U_i$ are pairwise disjoint, we obtain that the set $J:=\{ i\in I: K\cap U_i \not=\emptyset\}$ is finite. Therefore, for every $i\in I\setminus J$, the function $\mathbf{1}_{U_i}$ belongs to $[K;\e]$. As $J$ is finite, it follows that $\KK$ is a compact subset of $\CC(X)$.

It remains to show that $\KK$ is not equicontinuous at the point $z$. Set $\e:= 1/2$. Then, for every neighborhood $U$ of $z$, there are $i\in I$ and $x_i\in U_i$ such that $x_i\in U_i \cap U$ (this is possible since $z\in \overline{\bigcup\UU}$). Hence $|\mathbf{1}_{U_i}(x_i)-\mathbf{1}_{U_i}(z)|=1>\e$. Therefore the compact set $\KK$ is not equicontinuous. Thus $X$ is not an Ascoli space. \qed
\end{proof}
A concrete example of a non-discrete Tychonoff $P$-space is the one-point Lindel\"{o}fication of an uncountable discrete space.

To obtain concrete examples of non-$c_0$-quasibarrelled spaces we shall use the next assertion.

\begin{proposition} \label{p:countable-non-sequen-Ascoli}
Let $X$ be a countable non-discrete space whose compact subsets are finite. Then $X$ is not a sequentially Ascoli space.
\end{proposition}

\begin{proof}
Note that $X$ is zero-dimensional by Corollary 6.2.8 of \cite{Eng}. Let $z$ be a non-isolated point of $X$. Choose a maximal (under inclusion) family  $\UU=\{U_i: i\in I\}$ of clopen subsets of $X\setminus\{z\}$ such that $U_i\cap U_j=\emptyset$ for all distinct $i,j\in I$. The maximality of $\UU$ implies that $z\in \overline{\bigcup\UU}$. Moreover, since $X$ is countable, we have $I=\NN$.  Set $S:= \{ 0\} \cup \{ \mathbf{1}_{U_n}: n\in \NN\}$. It is clear that $\mathbf{1}_{U_n} \to 0$ in $\CC(X)=C_p(X)$ (the last equality holds because all comact subsets of $X$ are finite). To show that $S$ is not equicontinuous at the point $z$, let $\e=1/2$. Then for every neighborhood $U$ of $z$ there are $m\in\NN$ and $x_m\in U_m$ such that $x_m\in U$. Hence $|\mathbf{1}_{U_m}(x_m) - \mathbf{1}_{U_m}(z)|=1 >\e$. Thus the convergent sequence $S$ is not equicontinuous. \qed
\end{proof}

To show that a space is not  sequentially Ascoli we shall use the following proposition (we omit its proof since it is actually a partial case of (the proof of) Proposition 2.1 of \cite{GKP} when the index set $I$ is $\w$).
\begin{proposition} \label{p:sequentially-Ascoli-sufficient}
Assume  that a Tychonoff space $X$ admits a countable family $\U =\{ U_i : i\in \w\}$ of open subsets of $X$, a subset $A=\{ a_i : i\in \w\} \subseteq X$ and a point $z\in X$ such that
\begin{enumerate}
\item[{\rm (i)}] $a_i\in U_i$ for every $i\in \w$;
\item[{\rm (ii)}] $\big|\{ i\in \w : C\cap U_i\not=\emptyset \}\big| <\infty$  for each compact subset $C$ of $X$;
\item[{\rm (iii)}] $z$ is a cluster point of $A$.
\end{enumerate}
Then $X$ is not a sequentially Ascoli space.
\end{proposition}

Let $I$ be a partially ordered set. A family $\AAA=\{ A_i\}_{i\in I}$ of subsets of a set $\Omega$ is called {\em $I$-increasing} if $A_i \subseteq A_j$ for every $i\leq j$ in $I$. We say that the family $\AAA$ {\em swallows } a family $\mathcal{B}$ of subsets of $\Omega$ if for every $B\in \mathcal{B}$ there is an $i\in I$ such that $B\subseteq A_i$.
An $\NN$-increasing (respectively, $\NN^\NN$-increasing) family  of functionally bounded subsets of a topological space $X$ is called a {\em fundamental functionally bounded sequence} (respectively,  {\em fundamental functionally bounded  resolution}) in $X$ if it swallows the family of all functionally bounded subsets of $X$.
Analogously, an $\NN$-increasing (respectively, $\NN^\NN$-increasing) family of bounded subsets of an lcs $E$ is called a {\em fundamental bounded sequence} (respectively,  {\em  fundamental  bounded  resolution}) in $E$ if it swallows the family of all bounded subsets of $E$. Below we shall use the following assertion.

\begin{proposition} \label{p:ffbs-ffbr-L(X)}
Let $X$ be a Tychonoff space. Then:
\begin{enumerate}
\item[{\rm (i)}] $L(X)$ has a fundamental bounded sequence if and only if $X$ has a fundamental functionally bounded sequence;
\item[{\rm (ii)}] $L(X)$ has a fundamental bounded resolution if and only if $X$ has a fundamental functionally bounded resolution.
\end{enumerate}
\end{proposition}

\begin{proof}
We prove only (ii) because (i) can be proved analogously.
Let $\{ B_\alpha\}_{\alpha\in\NN^\NN}$ be a fundamental bounded resolution in $L(X)$. For every $\alpha\in\NN^\NN$, set $A_\alpha:=\supp(B_\alpha)$. We claim  that $\AAA=\{ A_\alpha\}_{\alpha\in\NN^\NN}$ is  a fundamental functionally bounded resolution in $X$. Indeed, by Proposition \ref{p:bounded-in-L(X)},   all sets $A_\alpha$ are functionally bounded in $X$ and, clearly, the family $\AAA$ is $\NN^\NN$-increasing. To show that $\AAA$ swallows  functionally bounded subsets of $X$, suppose for a contradiction that there exists a functionally bounded subset $A$ in $X$ such that $A\not\subseteq A_\alpha$ for every $\alpha\in\NN^\NN$.  Set
\[
B:=\{ \chi\in L(X): \; \supp(\chi)\subseteq A \mbox{ and } \|\chi\| \leq 1\}.
\]
Then $B$ is bounded in $L(X)$ by Proposition \ref{p:bounded-in-L(X)}. Therefore there is $\beta\in \NN^\NN$ such that $B\subseteq B_\beta$. But then $A=\supp(B)\subseteq \supp(B_\beta)=A_\beta$, a contradiction.

Conversely, let $\{ A_\alpha\}_{\alpha\in\NN^\NN}$ be a fundamental functionally bounded resolution in $X$. For every $\alpha=\big( \alpha(n)\big)\in\NN^\NN$, set
\[
B_\alpha:= \{ \chi\in L(X): \; \supp(\chi)\subseteq  A_\alpha \mbox{ and } \|\chi\| \leq \alpha(1)\}.
\]
Then,  by Proposition \ref{p:bounded-in-L(X)}, $B_\alpha$ is bounded in $L(X)$ for every $\alpha\in \NN^\NN$. We claim  that $\mathcal{B}=\{ B_\alpha\}_{\alpha\in\NN^\NN}$ is  a fundamental bounded resolution in $L(X)$. Indeed, it is clear that $\mathcal{B}$ is $\NN^\NN$-increasing. To show that $\mathcal{B}$ swallows the bounded sets of $L(X)$, let $B$ be a bounded subset of $L(X)$. Then, by Proposition \ref{p:bounded-in-L(X)}, $\supp(B)$ is functionally bounded in $X$ and $C_B =\sup\{ \| \chi\|: \chi\in B\}$ is finite. Choose $\alpha\in \NN^\NN$ such that $\supp(B) \subseteq A_\alpha$ and take $m\in \NN$ such that $m>C_B$. Set $\beta:= \alpha+(m)$, where $(m)$ is the constant sequence with entries $m$. Then $B\subseteq B_\beta$ by the definition of $B_\beta$. Thus $\mathcal{B} $ is  a fundamental bounded resolution in $L(X)$. \qed
\end{proof}


\section{Proof of main results} \label{seq:Main}


Let us recall some basic definitions. A {\em barrel} in an lcs $E$ is an absorbing absolutely convex closed set; an {\em $\aleph_0$-barrel} is a barrel $U$ that is the intersection of a sequence $\{ U_n\}_{n\in\w}$ of absolutely convex closed neighborhoods of zero.
An lcs $E$ is
\begin{enumerate}
\item[$\bullet$] {\em barrelled} if every barrel is a neighborhood of zero;
\item[$\bullet$] {\em $\aleph_0$-barrelled} if every $\aleph_0$-barrel is a neighborhood of zero;
\item[$\bullet$] {\em $\aleph_0$-quasibarrelled} if every $\beta(E',E)$-bounded subset of $E'$ which is the countable union of equicontinuous sets is itself equicontinuous;
\item[$\bullet$] {\em $\ell_\infty$-barrelled} if every $\sigma(E',E)$-bounded sequence is equicontinuous;
\item[$\bullet$] {\em $\ell_\infty$-quasibarrelled} if every $\beta(E',E)$-bounded sequence is equicontinuous;
\item[$\bullet$] {\em $c_0$-barrelled} if every $\sigma(E',E)$-null sequence is equicontinuous;
\item[$\bullet$] {\em $c_0$-quasibarrelled} if every $\beta(E',E)$-null sequence is equicontinuous;
\item[$\bullet$] {\em distinguished} if its strong dual is barrelled.
\end{enumerate}

The next proposition is crucial to prove Theorem \ref{t:L(X)-c0-barrelled}.
\begin{proposition} \label{p:Mackey-space-L(X)-0}
Let $X$ be a Tychonoff non-discrete space such that there exist a point $z\in X$, a sequence $\{ g_i\}_{i\in\w}$ of continuous functions from $X$ to $[0,2]$ and a sequence $\{ U_i\}_{i\in\w}$ of open subsets of $X$ such that
\begin{enumerate}
\item[{\rm (i)}] $\supp(g_i) \subseteq U_i $ for every $i\in\w$;
\item[{\rm (ii)}] $U_i\cap U_j=\emptyset $ for all distinct $i,j\in\w$;
\item[{\rm (iii)}] $z\not\in U_i$ for every $i\in\w$ and $z\in \cl\big(\bigcup_{i\in\w} \{ x\in X: g_i(x)\geq 1\}\big)$.
\end{enumerate}
Then the sequence $\{ g_i\}_{i\in\w}$ is a $\sigma(C(X),L(X))$-null and $\beta(C(X),L(X))$-bounded non-equicontinuous subset of the dual space $C(X)$ of $L(X)$.
\end{proposition}

\begin{proof}
It will be convenient to show that the sequence $S=\{ 2g_i\}_{i\in\w}$ satisfies the conclusion of the proposition.
It easily follows from (i) and (ii) that $2g_i\to 0$ in the pointwise topology, and hence $S$ is $\sigma(C(X),L(X))$-null. To show that $S$ is $\beta(C(X),L(X))$-bounded, fix an arbitrary bounded subset $M$ of $L(X)$. By Proposition \ref{p:bounded-in-L(X)}, the number $C:=\sup\{ \| \chi\|: \chi\in M\} $ is finite. Set $\lambda:= 1/(4C+4)$. Then, for every $i\in\w$ and each $\chi =a_1 x_1 +\cdots + a_m x_m\in M$ with distinct $x_1,\dots,x_m$ and nonzero $a_1,\dots,a_m\in\IR$, we obtain
\[
|\lambda 2g_i(\chi)|=\lambda |a_1 2g_i(x_1) +\cdots + a_m 2g_i(x_m)|\leq \lambda \cdot 4\big( |a_1|+\cdots+|a_m|\big) \leq \lambda \cdot 4 C<1,
\]
and hence  $\lambda 2g_i\in M^\circ$. Thus $S$ is $\beta(C(X),L(X))$-bounded.

To show that $S$ is not equicontinuous, fix a pointwise bounded and equicontinuous subset $A$ of $C(X)$ and $\e>0$. We have to show that $S\not\subseteq [A;\e]^\circ$. Observe that
\begin{equation} \label{equ:barrelledness-L(X)-1}
[A;\e]=\{ \chi\in L(X): |\chi(f)|<\e \; \mbox{ for all } f\in A\}.
\end{equation}
Since $A$ is equicontinuous, there is a neighborhood $U$ of $z$ such that
\begin{equation} \label{equ:barrelledness-L(X)-2}
|f(x)-f(z)|<\e \; \mbox{ for all } f\in A \mbox{ and } x\in U.
\end{equation}
 By (iii), there are $i\in\w$ and $x_i\in X$ such that $x_i\in U$ and $g_i(x_i)\geq 1$. Set $\eta:= x_i -z$. Then  (\ref{equ:barrelledness-L(X)-2}) implies
\[
|\eta(f)|=|f(x_i)-f(z)|<\e \; \mbox{ for all } f\in A,
\]
and hence, by (\ref{equ:barrelledness-L(X)-1}), $\eta\in [A;\e]$. But (i) and (iii) imply
\[
|2g_i(\eta)|=|2g_i(x_i)- 2g_i(z)|=2g_i(x_i)\geq 2.
\]
Thus $2g_i\not\in [A;\e]^\circ$ and hence $S$ is not equicontinuous. \qed
\end{proof}

Now we are ready to prove the first main result of the paper.

\smallskip
{\em Proof of Theorem \ref{t:L(X)-c0-barrelled}.}
We shall consider continuous functions on $X$ also as  continuous functionals of $L(X)$.
The implications  (i)$\Rightarrow$(ii), (i)$\Rightarrow$(iii)  and (iv)$\Rightarrow$(ii)  are clear.

\smallskip
(ii)$\Rightarrow$(v) and (iii)$\Rightarrow$(v) follow from the following two claims.

\smallskip
{\em Claim 1. The space $X$ is zero-dimensional.} Indeed, suppose for a contradiction that $X$ is not zero-dimensional. Then, by (the proof of) Proposition 2.4 of \cite{Gabr-L(X)-Mackey},  there exist a point $z\in X$, a sequence $\{ g_i\}_{i\in\w}$ of continuous functions from $X$ to $[0,2]$ and a sequence $\{ U_i\}_{i\in\w}$ of open subsets of $X$ which satisfy (i)-(iii) of Proposition \ref{p:Mackey-space-L(X)-0}. Now  Proposition \ref{p:Mackey-space-L(X)-0} implies that  the sequence $\{ g_i\}_{i\in\w}$, being  $\sigma(C(X),L(X))$-null and $\beta(C(X),L(X))$-bounded, is not equicontinuous. Thus $L(X)$ is neither $c_0$-barrelled nor $\ell_\infty$-quasibarrelled, a contradiction.

\smallskip
{\em Claim 2. The zero-dimensional space $X$ is a $P$-space.} Suppose for a contradiction that $X$ is not a $P$-space.
Then there exist a point $x_0\in X$ and a sequence $\{ V_i\}_{i\in\w}$ of open  neighborhoods of $x$ such that $\bigcap_{n\in\w} V_n$ is not a neighborhood of $x_0$.
 Since $X$ is zero-dimensional and $\bigcap_{n\in\w} V_n$ is not a neighborhood of $x_0$, we can assume that all $V_n$ are clopen neighborhoods of $x_0$ and $V_{n+1}$ is a proper subset of $V_n$ for every $n\in\w$. For every $n\in\w$, set
\[
U_n := V_n\setminus V_{n+1} \; \mbox{ and } \; g_n :=  \mathbf{1}_{U_n}.
\]
 Since $\bigcup_{n\in\w} U_n = V_0 \setminus \bigcap_{n\in\w} V_n$, we obtain $x_0 \in \overline{ \bigcup_{n\in\w} U_n}$. Therefore the point $x_0$, and the sequences $\{ U_n\}_{n\in\w}$ and $S=\{ g_n\}_{n\in\w}$ satisfy (i)-(iii) of Proposition \ref{p:Mackey-space-L(X)-0}. Therefore  the sequence $S$, being  $\sigma(C(X),L(X))$-null and $\beta(C(X),L(X))$-bounded, is  not equicontinuous. Thus $L(X)$ is neither $c_0$-barrelled nor $\ell_\infty$-quasibarrelled, a contradiction.

\smallskip
(v)$\Rightarrow$(i) 
Fix an arbitrary $\sigma\big( C(X),L(X)\big)$-bounded sequence $S=\{ f_n\}_{n\in\w}$ in the dual space $L(X)'=C(X)$ of $L(X)$. Then, by (i) of Proposition \ref{p:P-space-Ascoli}, $S$ is a pointwise bounded and equicontinuous subset of $C(X)$. Therefore, by Theorem  \ref{t:topology-L(X)-Raikov}, $[S;1]$ is a neighborhood of zero in $L(X)$. It is clear that $S\subseteq [S;1]^\circ$. Thus $S$ is equicontinuous in $L(X)'$.

\smallskip
(v)$\Rightarrow$(iv) Let $A$ be a  $\beta(C(X),L(X))$-bounded subset of $L(X)'=C(X)$ which is the union of a sequence $\{ A_n\}_{\in\w}$ of equicontinuous subsets of $C(X)$. We have to show that $A$ is equicontinuous. 
Since $A$ is strongly bounded 
and $X$ is a subspace of $L(X)$, we obtain that $A$ is a pointwise bounded subset of $C(X)$. Let us show that $A$ is also an equicontinuous subset of $C(X)$. Fix an arbitrary point $z\in X$ and $\e>0$. Since all $A_n$ are equicontinuous subsets of the dual of $L(X)$, Proposition \ref{p:equicontinuity-C(X)} implies that all $A_n$ are equicontinuous subsets of the function space $C(X)$. Therefore, for every $n\in\w$, there is an open neighborhood $U_n$ of $z$ such that
\begin{equation} \label{equ:L(X)-aleph-quasibarrelled-1}
|f(x)-f(z)|<\e, \quad \mbox{ for all } x\in U_n \mbox{ and } f\in A_n.
\end{equation}
As $X$ is a $P$-space, the set $U:= \bigcap_{n\in\w} U_n$ is an open neighborhood of $z$, and  (\ref{equ:L(X)-aleph-quasibarrelled-1}) implies
\[
|f(x)-f(z)|<\e, \quad \mbox{ for all } x\in U \mbox{ and } f\in A.
\]
Thus $A$ is equicontinuous and hence, by Theorem \ref{t:topology-L(X)-Raikov}, the set $[A;1]$ is a neighborhood of zero in $L(X)$. Finally, the Alaoglu theorem and Theorem~11.11.5 of \cite{NaB} imply that $[A;1]^{\circ}$ and hence $A$ are $\beta(C(X),L(X))$-bounded.  \qed


\medskip

Let $X$ be  a Tychonoff space. A subset $S$ of $C(X)$ is called {\em uniformly bounded on a subset $A$ of $X$} if there is $C_S>0$ such that $|f(x)|\leq C_S$ for every $x\in A$ and each $f\in S$. Below we describe strongly bounded subsets of the dual $C(X)$ of $L(X)$.
\begin{proposition} \label{p:strongly-bounded-in-C(X)}
A subset $S$ of $C(X)$ is $\beta\big( C(X),L(X)\big)$-bounded if and only if $S$ is uniformly bounded on every functionally bounded subset of $X$.
\end{proposition}

\begin{proof}
Assume that $S$ is $\beta\big( C(X),L(X)\big)$-bounded and let $A$ be a functionally bounded subset of $X$. Set
\[
B:=\{ \chi\in L(X): \supp(\chi)\subseteq A \; \mbox{ and }\; \|\chi\| \leq 1\}.
\]
Then, by Proposition \ref{p:bounded-in-L(X)},  $B$ is a bounded subset of $L(X)$. Therefore there exists $\lambda>0$ such that $S\subseteq \lambda B^\circ$. Since every $x\in A$ considered as an element of $L(X)$ belongs to $B$, we obtain
\[
|f(x)|\leq \lambda \; \mbox{ for every } f\in S.
\]
Thus $S$ is uniformly bounded on $A$.

Conversely, assume that $S$ is uniformly bounded on every functionally bounded subset of $X$. Fix a bounded subset $B$ of $L(X)$. Then,  by Proposition \ref{p:bounded-in-L(X)},  the set $A:=\mathrm{supp}(B)$ is functionally bounded in $X$ and the number $C_B=\sup\{ \| \chi\|: \chi\in B\}$ is finite. Therefore there is $\lambda >0$ such that $|f(x)|\leq \lambda$ for every $x\in A$ and each $f\in S$. Hence
\[
|f(\chi)|\leq \lambda \cdot \|\chi\| \leq \lambda \cdot C_B, \; \mbox{ for every } \chi\in B \mbox{ and } f\in S.
\]
Therefore $S\subseteq (\lambda C_B) \cdot B^\circ$. Since $B$ was arbitrary, we obtain that $S$ is a strongly bounded subset of $C(X)$. \qed
\end{proof}

\begin{proposition} \label{p:strong-dual-of-L(X)}
For every Tychonoff space $X$,   the restriction map $R: L(X)'_\beta \to \CC(X)$, $R(F):= F|_X$, is a continuous isomorphism. If in addition  $X$ is a $\mu$-space, then $R$ is a topological isomorphism, and hence $L(X)$ is distinguished.
\end{proposition}

\begin{proof}
It is well known (and easy to show) that $R$ is a linear isomorphism. To prove that $R$ is continuous, fix a standard neighborhood $[K;\e]$ of zero in $\CC(X)$. Define
\[
B:=\left\{ \chi\in L(X): \supp(\chi)\subseteq K \mbox{ and } \| \chi\| \leq \frac{2}{\e} \right\}.
\]
Then, by Proposition \ref{p:bounded-in-L(X)},  $B$ is a bounded subset of $L(X)$. Since $(2/\e)x \in B$ for every $x\in K$, it follows that
\[
\left| \frac{2}{\e} g(x)\right| = \left| g\left(\frac{2}{\e} x\right)\right| \leq 1, \quad \mbox{ for every } g\in B^\circ.
\]
Therefore $R(B^\circ) \subseteq [K;\e]$ and hence $R$ is continuous.

Assume that  $X$ is a $\mu$-space. To show that $R$ is also open, fix a nonzero bounded subset $A$ of $L(X)$. Then, by Proposition \ref{p:bounded-in-L(X)}, $\supp(A)$ is functionally bounded in $X$ and the number $C_A=\sup\{ \|\chi\|: \chi\in A\}$ is finite. Define $K:=\overline{\supp(A)}$ and $\e:= 1/C_A$. Since $X$ is a $\mu$-space, $K$ is a compact subset of $X$. Hence, for each $f\in [K;\e]$ and for every   $\chi = a_1 x_1+\cdots +a_n x_n\in A$ with distinct $x_1,\dots, x_n\in \supp(A)$ and  nonzero $a_1,\dots,a_n\in\IR$, we obtain
\[
|f(\chi)|=| a_1 f(x_1)+\cdots +a_n f(x_n)|\leq \e ( |a_1|  +\cdots + |a_n|)\leq \e \cdot C_A =1.
\]
Therefore $f\in R(A^\circ)$ and hence $[K;\e] \subseteq R(A^\circ)$. Thus $R$ is open.

 As we proved $L(X)'_\beta =\CC(X)$. Since $X$ is a $\mu$-space, the Nachbin--Shirota theorem implies that $\CC(X)$ is barrelled.  Thus $L(X)$ is distinguished. \qed
\end{proof}

Below we prove Theorem \ref{t:L(X)-c0-quasibarrelled}.

\smallskip
{\em Proof of Theorem \ref{t:L(X)-c0-quasibarrelled}}.
Fix a $\beta\big( C(X),L(X)\big)$-null sequence $S=\{ f_n\}_{n\in\w}$ in $C(X)$. We have to show that $S$ is equicontinuous in $C(X)=L(X)'$. Proposition \ref{p:strong-dual-of-L(X)} implies that the strong dual $L(X)'_\beta =\big( C(X), \beta\big( C(X),L(X)\big)\big)$ of $L(X)$ is topologically isomorphic to $\CC(X)$. Since $X$ is sequentially Ascoli,  $S$ is an equicontinuous sequence of continuous functions on $X$. Clearly, $S$ is pointwise bounded and hence, by Theorem \ref{t:topology-L(X)-Raikov}, $[S;1]$ is a neighborhood of zero in $L(X)$. Now the inclusion $S\subseteq [S;1]^\circ$ implies that $S$ is equicontinuous as a subset of the dual space of $L(X)$. Thus $L(X)$ is a $c_0$-quasibarrelled space. \qed

\medskip

The next assertion provides a sufficient condition on $X$ such that  $L(X)$ is not a $c_0$-quasibarrelled space.
\begin{proposition} \label{p:L(X)-not-c0-quasibarrelled}
Let $X$ be a $\mu$-space which admits a countable family $\U =\{ U_i : i\in \w\}$ of open subsets of $X$, a subset $A=\{ a_i : i\in \w\} \subseteq X$ and a point $z\in X$ satisfying (i)-(iii) of Proposition \ref{p:sequentially-Ascoli-sufficient}. Then $L(X)$ is not a $c_0$-quasibarrelled space.
\end{proposition}

\begin{proof}
By Proposition \ref{p:sequentially-Ascoli-sufficient}, the space $\CC(X)$ has a convergent sequence $S=\{ f_n\}_{n\in\w}$ which is not equicontinuous.
Proposition \ref{p:strong-dual-of-L(X)} implies that the strong dual of $L(X)$ is topologically isomorphic to $\CC(X)$, and hence $S$ is strongly null. On the other hand, by Proposition \ref{p:equicontinuity-C(X)}, $S$ is not equicontinuous as a subset of the dual space of $L(X)$. Thus $L(X)$ is not a $c_0$-quasibarrelled space. \qed
\end{proof}

\begin{example} \label{exa:non-c0-quasibarrelled-1} {\em
Let $X$ be an infinite-dimensional Banach space endowed with the weak topology. Then, by (the proof of) Proposition 3.2, respectively, the space $X$ admits a countable family $\U$, a countable subset $A$ and a point $z\in X$ satisfying conditions of Proposition \ref{p:L(X)-not-c0-quasibarrelled}. A result of Valdivia \cite{Valdivia-77} (which states that every quasi-complete space $E$ in the weak topology is a $\mu$-space) implies that $X$ is a $\mu$-space. Thus, by Proposition \ref{p:L(X)-not-c0-quasibarrelled}, $L(X)$ is not a $c_0$-quasibarrelled space.} \qed
\end{example}


Following \cite{FGK-quasi-DF}, an lcs $(E,\tau)$ is called a {\em quasi-$(DF)$-space} if $E$ admits a fundamental bounded resolution and belongs to the class $\mathfrak{G}$.
For the definition of the class $\mathfrak{G}$, we refer the reader to the original paper \cite{CasOr}   by Cascales and Orihuela or to the book \cite[Section~11]{kak}. 
Below we prove Theorem \ref{t:L(X)-DF-space}.

\smallskip
{\em Proof of Theorem \ref{t:L(X)-DF-space}}.
The implication (i)$\Rightarrow$(ii) is trivial. To prove the implication (ii)$\Rightarrow$(iii) assume that $L(X)$ is an almost $(DF)$-space. Then $X$ has a fundamental functionally bounded sequence $\AAA=\{ A_n\}_{n\in\NN}$ by Proposition \ref{p:ffbs-ffbr-L(X)} and is a $P$-space by Theorem \ref{t:L(X)-c0-barrelled}. Now Lemma \ref{l:bounded-in-P-space} implies that all $A_n$ are finite. Therefore $X$ is countable. Since every countable $P$-space is discrete, we obtain that $X$ is a countable discrete space.

(iii)$\Rightarrow$(i) If $X$ is finite, then $L(X)=\IR^{|X|}$ is trivially a $(DF)$-space. If $X$ is countably infinite and hence $L(X)=\phi$, then $L(X)$ is a $(DF)$-space by Theorem 12.4.8 of \cite{Jar}. \qed

\smallskip

\begin{example} \label{exa:non-c0-quasibarrelled}
There is a countable Tychonoff space $X$ such that  $L(X)$ is a  quasi-$(DF)$-space, but it is not a $c_0$-quasibarrelled space.
\end{example}

\begin{proof}
Let $X$ be the countable Tychonoff space defined in Example~4.10 of \cite{FGK-quasi-DF}. Then, as we proved in \cite{FGK-quasi-DF}, $L(X)$ is a  quasi-$(DF)$-space. To show that $L(X)$ is not $c_0$-quasibarrelled, we note first that any countable space is Lindel\"{o}f and hence a $\mu$-space.
Therefore, by Proposition \ref{p:strong-dual-of-L(X)},  the strong dual space of $L(X)$ is topologically isomorphic to $\CC(X)$. Since every compact subset of $X$ is finite (see \cite[Example~4.10]{FGK-quasi-DF}),  Proposition \ref{p:countable-non-sequen-Ascoli} implies that $\CC(X)$ has a convergent sequence $S$ which is not equicontinuous as a subset of $C(X)$. Hence, by Proposition \ref{p:equicontinuity-C(X)},  $S$ is not equicontinuous as a subset of the dual space of $L(X)$. Thus $L(X)$ is not a $c_0$-quasibarrelled space. \qed
\end{proof}

In particular, the space $L(X)$ from Example \ref{exa:non-c0-quasibarrelled}  is not a $(df)$-space. On the other hand, there are also $(df)$-spaces $L(X)$ which are not quasi-$(DF)$-spaces.

\begin{proposition}  \label{p:df-non-quasi-DF}
If $K$ is a non-metrizable compact space, then $L(K)$ is a $(df)$-space but not a quasi-$(DF)$-space.
\end{proposition}

\begin{proof}
Theorem \ref{t:L(X)-c0-quasibarrelled} and Proposition \ref{p:ffbs-ffbr-L(X)} immediately imply that $L(K)$ is a $(df)$-space. Any compact subset of an lcs from the class $\GG$ is metrizable, see \cite{CasOr} or \cite[Theorem~11.1]{kak}. But since $K$, being a subspace of $L(K)$, is not metrizable we obtain that $L(K)$ is not a quasi-$(DF)$-space. \qed
\end{proof}

Now we prove Theorem \ref{t:L(X)-Grothendieck}.

\smallskip

{\em Proof of Theorem \ref{t:L(X)-Grothendieck}.}
Assume that $L(X)$ has the Grothendieck property and suppose for a contradiction that $X$ has an infinite compact subset $K$. Using  Lemma 11.7.1 of \cite{Jar}, we choose a one-to-one sequence $S=\{ a_n\}_{n\in\NN}$ in $K$ and a sequence $\{ U_n\}_{n\in\NN}$ of open pairwise disjoint subsets of $X$ such that $a_n\in U_n$ for every $n\in\NN$. Since $K$ is compact, $S$ has a cluster point $z\in K$. Without loss of generality we assume that $z\not\in S$. For every $n\in\NN$, choose a continuous function $f_n:X\to [0,2^n]$ such that
\[
\supp(f_n) \subseteq U_n \setminus \{ z\} \; \mbox{ and } \; f_n(a_n)=2^n.
\]
Since $U_n$ are pairwise disjoint, it follows that $f_n\to 0$ in the pointwise topology and hence in $\sigma\big( C(X),L(X)\big)$. Set $\mu: =\sum_{n\in\NN} 2^{-n} \delta_{a_n}$. Then, by Proposition \ref{p:strong-dual-of-L(X)}, $\mu\in \CC(X)'= L(X)''$ and $\mu(f_n)=1$ for every $n\in\NN$. Therefore $f_n\not\to 0$ in the weak topology of $\CC(X)$. Thus $L(X)$ does not have the Grothendieck property, a contradiction.

Conversely, if all compact subsets of $X$ are finite, then, by Proposition \ref{p:strong-dual-of-L(X)}, $L(X)''=\CC(X)'=C_p(X)'=L(X)$ and hence $L(X)$ trivially has the Grothendieck property.
\qed

\smallskip

To prove Theorem \ref{t:L(X)-DP-property} we shall use the following characterization of the $(DP)$ property.

\begin{theorem}[\protect{\cite[Theorem~9.3.4]{Edwards}}] \label{t:Edwards-DP}
An lcs $E$ has the $(DP)$ property if and only if every absolutely convex, weakly compact subset of $E$ is precompact for the topology $\tau_{\Sigma'}$ of uniform convergence on the absolutely convex, equicontinuous, weakly compact subsets of $E'_\beta$.
\end{theorem}

\medskip

{\em Proof of Theorem \ref{t:L(X)-DP-property}.}
First we prove the following claim.

{\em Claim 1. The topology $\tau_{\Sigma'}$ defined in Theorem \ref{t:Edwards-DP} is weaker than the original topology $\pmb{\nu}_X$ of $L(X)$}. Indeed, let $A$ be an absolutely convex, equicontinuous, weakly compact subset of $L(X)'_\beta$. Then, by Proposition \ref{p:equicontinuity-C(X)}, $A$ is an equicontinuous family of continuous functions on $X$. Taking into account that the weak topology of $L(X)'_\beta$ is stronger than the weak-$\ast$ topology $\sigma\big( C(X),L(X)\big)$ on $C(X)$, the weak compactness of $A$ implies that $A$ is also pointwise bounded. Therefore, by Theorem \ref{t:topology-L(X)-Raikov}, the polar $A^\circ$ of $A$ in $L(X)$ is a neighborhood of zero. Thus $\tau_{\Sigma'}\leq \pmb{\nu}_X$ and the claim is proved.

Now let $K$ be an absolutely convex, weakly compact subset of $L(X)$. Then, by Theorem 1.2 of \cite{Gab-Respected}, $K$ is a compact subset of $L(X)$. Therefore, by Claim 1, $K$ is even compact for the topology $\tau_{\Sigma'}$. Finally, Theorem \ref{t:Edwards-DP}  implies that $L(X)$ has the $(DP)$-property.
\qed 

\medskip
Recall that a topological space $X$ is called {\em sequential} if for each non-closed subset $A\subseteq X$ there is a sequence $\{a_n\}_{n\in\w}\subseteq A$ converging to some point $a\in \overline{A}\setminus A$. We note that a sequential space $X$ is discrete if and only if every convergent sequence is eventually constant.


\medskip

{\em Proof of Theorem \ref{t:L(X)-sDP-property}.}
Assume that $L(X)$ has the $(sDP)$-property and suppose for a contradiction that $X$ is not discrete. Then $X$ contains a one-to-one sequence $S=\{ a_n\}_{n\in\NN}$ which converges to some element $a_0 \not\in S$. Using  Lemma 11.7.1 of \cite{Jar} and passing to a subsequence if needed, we can choose a  sequence $\{ U_n\}_{n\in\NN}$ of open pairwise disjoint subsets of $X$ such that $a_n\in U_n$ and $a_0\not\in U_n$ for every $n\in\NN$. For every $n\in\NN$, choose a continuous function $f_n:X\to [0,1]$ with support in $U_n$ and such that $f_n(a_n)=1$. Clearly, $f_n\to 0$ in the pointwise topology on $C(X)$.

We claim that $f_n \to 0$ in the weak topology of $L(X)'_\beta$. Indeed, since $L(X)'_\beta=\CC(X)$ by Proposition \ref{p:strong-dual-of-L(X)}, we obtain $L(X)''=M_c(X)$. For every $\mu\in M_c(X)$, the  Lebesgue dominated convergence theorem implies that $\mu(f_n)\to 0$. The claim is proved.

Now, for every $n\in\NN$, set $\eta_n := a_n - a_0$. As $a_n \to a_0$ in $X$ we obtain $\eta_n \to 0$ in $L(X)$. Since
\[
f_n (\eta_n)=f_n(a_n) - f_n(a_0)=1 \not\to 0
\]
we obtain that $L(X)$ does not have the $(sDP)$-property, a contradiction. Thus $X$ must be discrete.

Conversely, assume that $X$ is discrete. Then $L(X)$ is the direct sum of $|X|$-many copies of $\IR$, and hence the strong dual of $L(X)$ is $\IR^{|X|}$. Since $\IR^{|X|}$ has the Schur property, Proposition 3.1 of \cite{Gabr-free-resp} implies that $L(X)$ has the $(sDP)$-property. \qed

\medskip

\bibliographystyle{amsplain}

\end{document}